\documentclass[11pt,a4paper,reqno]{amsart}
\usepackage[english]{babel}
\usepackage[applemac]{inputenc}
\usepackage[T1]{fontenc}
\usepackage{palatino}
\usepackage{cite}
\usepackage{esint}
\usepackage{soul}
\usepackage{mathabx}
\usepackage{verbatim}
\usepackage{amsmath}
\usepackage{amssymb}
\usepackage{amsthm}
\usepackage{amsfonts}
\usepackage{graphicx}
\usepackage{mathtools}
\usepackage{enumitem}

\usepackage[colorlinks = true, citecolor = black]{hyperref}
\title[Absolute continuity and Alberti representations]{Quantitative absolute continuity of planar measures with two independent Alberti representations}
\author{David Bate and Tuomas Orponen}
\address{University of Helsinki, Department of Mathematics and Statistics}
\thanks{D.B. is supported by the Academy of Finland via the project \emph{Projections, densities and rectifiability: new settings for classical ideas}, grant No. 308510. T.O. is supported by the Academy of Finland via the project \emph{Quantitative rectifiability in Euclidean and non-Euclidean spaces}, grant No. 309365.}
\email{david.bate@helsinki.fi \\ tuomas.orponen@helsinki.fi}

\newcommand{\R}{\mathbb{R}}

\newcommand{\N}{\mathbb{N}}

\newcommand{\tn}{\mathbb{P}}

\newcommand{\calL}{\mathcal{L}}
\newcommand{\calD}{\mathcal{D}}
\newcommand{\calH}{\mathcal{H}}
\newcommand{\calC}{\mathcal{C}}

\newcommand{\spt}{\operatorname{spt}}

\newcommand{\calP}{\mathcal{P}}

\newcommand{\spa}{\operatorname{span}}
\newcommand{\diam}{\operatorname{diam}}

\newcommand{\dist}{\operatorname{dist}}

\numberwithin{equation}{section}

\theoremstyle{plain}
\newtheorem{thm}[equation]{Theorem}

\newtheorem{lemma}[equation]{Lemma}

\newtheorem{ex}[equation]{Example}
\newtheorem{cor}[equation]{Corollary}
\newtheorem{proposition}[equation]{Proposition}
\newtheorem{question}{Question}

\theoremstyle{definition}

\newtheorem{definition}[equation]{Definition}

\newtheorem{notation}[equation]{Notation}

\theoremstyle{remark}
\newtheorem{remark}[equation]{Remark}

\newcommand{\nref}[1]{(\hyperref[#1]{#1})}

\begin{document}

\begin{abstract} We study measures $\mu$ on the plane with two independent Alberti representations. It is known, due to Alberti, Cs\"ornyei, and Preiss, that such measures are absolutely continuous with respect to Lebesgue measure. The purpose of this paper is to quantify the result of A-C-P. Assuming that the representations of $\mu$ are bounded from above, in a natural way to be defined in the introduction, we prove that $\mu \in L^{2}$. If the representations are also bounded from below, we show that $\mu$ satisfies a reverse H\"older inequality with exponent $2$, and is consequently in $L^{2 + \epsilon}$ by Gehring's lemma. A substantial part of the paper is also devoted to showing that both results stated above are optimal. \end{abstract}

\maketitle

\section{Introduction}

Before stating any results, we need to define a few key concepts.
\begin{definition}[Cones and $\calC$-graphs] A \emph{cone} stands for a subset of $\R^{d}$ of the form
\begin{displaymath} \calC = \calC(e,\theta) = \{x \in \R^{d} : |x \cdot e| \geq \theta |x|\}, \end{displaymath}
where $e \in S^{d - 1}$ and $0 < \theta \leq 1$. Given a cone $\calC \subset \R^{d}$, a \emph{$\calC$-graph} is any set $\gamma \subset \R^{d}$ such that
\begin{displaymath} x - y \in \calC \text{ for all } x,y \in \gamma. \end{displaymath}
A $\calC$-graph $\gamma$ is called \emph{maximal} if the orthogonal projection $\pi_{e} \colon \gamma \to \spa(e)$ is surjective. The family of all maximal $\calC$-graphs is denoted by $\Gamma_{\calC}$. We record here that if $\gamma$ is a $\calC(e,\theta)$-graph with $\theta > 0$, then the orthogonal projection $\pi_{e} \colon \gamma \to \spa(e)$ is a bilipschitz map. Also, if $\gamma$ is maximal, then $\calH^{1}|_{\gamma}$ is a $1$-regular measure on $\gamma$. In other words, there exist constants $0 < c \leq C < \infty$ such that $cr \leq \calH^{1}(\gamma \cap B(x,r)) \leq Cr$ for all $x \in \gamma$ and $r > 0$. 

We say that two cones $\calC_{1},\calC_{2}$ are \emph{independent} if they are angularly separated as follows:
\begin{equation}\label{independence} \tau := \inf\{\angle(x_{1},x_{2}) : x_{1} \in \calC_{1} \, \setminus \, \{0\} \text{ and } x_{2} \in \calC_{2} \, \setminus \, \{0\}\} > 0. \end{equation}
\end{definition}

\begin{definition}[Alberti representations]\label{albertiDef} Let $\calC \subset \R^{d}$ be a cone. Let $(\Omega,\Sigma,\tn)$ be a measure space with $\tn(\Omega) < \infty$, and let $\gamma \colon \Omega \to \Gamma_{\calC}$ be a map such that
\begin{equation}\label{measurability} \omega \mapsto \calH^{1}(B \cap \gamma(\omega)) \text{ is $\Sigma$-measurable} \end{equation} 
for all Borel sets $B \subset \R^{d}$. Then, the formula
\begin{equation}\label{nuDef} \nu(B) := \nu_{(\Omega,\tn,\gamma)}(B) := \int_{\Omega} \calH^{1}(B \cap \gamma(\omega)) \, d\tn(\omega) \end{equation}
makes sense for all Borel sets $B \subset \R^{d}$, and evidently $\nu(K) < \infty$ for all compact sets $K \subset \R^{d}$. We extend the definition to all sets $A \subset \R^{d}$ via the usual procedure of setting first $\nu^{\ast}(A) := \inf\{\nu(B) : A \subset B \text{ Borel}\}$. This process yields a Radon measure $\nu^{\ast}$ which agrees with $\nu$ on Borel sets. In the sequel, we just write $\nu$ in place of $\nu^{\ast}$. 

If $\mu$ is another Radon measure on $\R^{d}$, we say that \emph{$\mu$ is representable by $\calC$-graphs} if there is a triple $(\Omega,\tn,\gamma)$ as above such that $\mu \ll \nu_{(\Omega,\tn,\gamma)} =: \nu$. In this case, the quadruple $(\Omega,\tn,\gamma,\tfrac{d\mu}{d\nu})$ is an \emph{Alberti representation of $\mu$ by $\calC$-graphs}. The representation is
\begin{itemize}
\item \emph{bounded above} (BoA) if $\tfrac{d\mu}{d\nu} \in L^{\infty}(\nu)$,
\item \emph{bounded below} (BoB) if $(\tfrac{d\mu}{d\nu})^{-1} \in L^{\infty}(\nu)$.
\end{itemize}
We also consider local versions of these properties: the representation is BoA (resp. BoB) on a Borel set $B \subset \R^{2}$ if $\tfrac{d\mu}{d\nu} \in L^{\infty}(B,\nu)$ (resp. $(\tfrac{d\mu}{d\nu})^{-1} \in L^{\infty}(B,\nu)$). Two Alberti representations of $\mu$ by $\calC_{1}$- and $\calC_{2}$-graphs are \emph{independent}, if the cones $\calC_{1},\calC_{2}$ are independent in the sense \eqref{independence}.
\end{definition}

Representations of this kind first appeared in Alberti's paper \cite{Al} on the rank-$1$ theorem for $BV$-functions. It has been known for some time that planar measures with two independent Alberti representations are absolutely continuous with respect to Lebesgue measure; this fact is due to Alberti, Cs\"ornyei, and Preiss, see \cite[Proposition 8.6]{ACP}, but a closely related result is already contained in Alberti's original work, see \cite[Lemma 3.3]{Al}. The argument in \cite{ACP} is based on a decomposition result for null sets in the plane, \cite[Theorem 3.1]{ACP}. Inspecting the proof, the following statement can be easily deduced: if $\mu$ is a planar measure with two independent BoA representations, then $\mu \in L^{2,\infty}$. The proof of \cite[Lemma 3.3]{Al}, however, seems to point towards $\mu \in L^{2}$, and the first statement of Theorem \ref{main} below asserts that this is the case. Our argument is short and very elementary, see Section \ref{boundedReps}. The main work in the present paper concerns measures with two independent representations which are both BoA and BoB. In this case, Theorem \ref{main} asserts an $\epsilon$-improvement over the $L^{2}$-integrability.

\begin{thm}\label{main} Let $\mu$ be a Radon measure on $\R^{2}$ with two independent Alberti representations. If both representations are BoA, then $\mu \in L^{2}(\R^{2})$. If both representations are BoA and BoB on $B(2)$, then there exists a constant $C \geq 1$ such that $\mu$ satisfies the reverse H\"older inequality
\begin{equation}\label{reverseHolder} \left( \frac{1}{r^{2}} \int_{B(x,r)} \mu(x)^{2} \, dx \right)^{1/2} \leq \frac{C}{r^{2}} \int_{B(x,r)} \mu(x) \, dx, \qquad B(x,r) \subset B(1). \end{equation}
As a consequence, $\mu \in L^{2 + \epsilon}(B(\tfrac{1}{2}))$ for some $\epsilon > 0$.
\end{thm}

The final conclusion follows easily from Gehring's lemma, see \cite[Lemma 2]{Ge}. 

\subsection{Sharpness of the main theorem} We now discuss the sharpness of Theorem \ref{main}. For illustrative purposes, we make one more definition. Let $\mu$ be a Radon measure on $\R^{2}$. We say that $(\Omega,\tn,\gamma,\tfrac{d\mu}{d\nu})$ is an \emph{axis-parallel representation of $\mu$} if $\Omega = \R$, and $\gamma \colon \Omega \to \calP(\R^{2})$ is one of the two maps $\gamma_{x}(\omega) = \{\omega\} \times \R$ or $\gamma_{y}(\omega) = \R \times \{\omega\}$. Note that two axis-parallel representations $(\R,\tn_{1},\gamma_{1},\tfrac{d\mu}{d\nu_{1}})$ and $(\R,\tn_{2},\gamma_{2},\tfrac{d\mu}{d\nu_{2}})$ are independent if and only if $\{\gamma_{1},\gamma_{2}\} = \{\gamma_{x},\gamma_{y}\}$.

The following example shows that two independent BoA representations -- even axis parallel ones -- do not guarantee anything more than $L^{2}$:
\begin{ex}\label{mainEx} Fix $r > 0$ and consider the measure $\mu_{r} = \tfrac{1}{r} \cdot \mathbf{1}_{[0,r]^{2}}$. Note that $\|\mu_{r}\|_{L^{p}} = r^{2 - p}$ for $p \geq 1$, so $\mu_{r} \in L^{p}$ uniformly in $r > 0$ if and only if $p \leq 2$. On the other hand, consider the probability $\tn := \tfrac{1}{r} \cdot \calL^{1}|_{[0,r]}$ on $\Omega = \R$, and the maps $\gamma_{1} := \gamma_{x}$ and $\gamma_{2} := \gamma_{y}$, as above. Writing $\nu_{j} := \nu_{(\Omega,\tn,\gamma_{j})}$ for $j \in \{1,2\}$, it is easy to check that
\begin{displaymath} \|\mu_{r}\|_{L^{\infty}(\nu_{j})} \leq 1, \qquad j \in \{1,2\}. \end{displaymath}
So, $\mu_{r}$ has two independent axis-parallel BoA representations with constants uniformly bounded in $r > 0$. After this, it is not difficult to produce a single measure $\mu$ with two independent axis-parallel BoA representations which is not in $L^{p}$ for any $p > 2$: simply place disjoint copies of $c_{j}\mu_{r_{j}}$ along the diagonal $\{(x,y) : x = y\}$, where $\sum c_{j} = 1$ and $r_{j} \to 0$ rapidly.
\end{ex}

The situation where both representations are (locally) both BoA and BoB is more interesting. We start by recording the following simple proposition, which shows that Theorem \ref{main} is far from sharp for axis-parallel representations:
\begin{proposition}\label{axisParallel} Let $\mu$ be a finite Radon measure on $\R^{2}$ with two independent axis-parallel representations, both of which are BoA and BoB on $[0,1)^{2}$. Then there exist constants $0 < c \leq C < \infty$, depending only on the BoA and BoB constants, such that $\mu|_{[0,1)^{2}} = f \, d\calL^{2}|_{[0,1)^{2}}$, where $0 < c \leq f(x) \leq C < \infty$ for $\calL^{2}$ almost every $x \in [0,1)^{2}$.
\end{proposition}

We give the easy details in the appendix. In the light of the proposition, the following theorem is perhaps a little surprising:
\begin{thm}\label{main3} Let $0 < \alpha < 1$. The measure $\mu = f\, d\calL^{2}$, where
\begin{equation}\label{form23} f(x) = |x|^{-\alpha}\mathbf{1}_{B(1)}, \end{equation}
has two independent Alberti representations which are both BoA and BoB on $B(1)$. \end{thm}
The representations are, of course, not axis-parallel. For a picture, see Figure \ref{fig4}.  Since
\begin{displaymath} f \in L^{p}(B(1)) \quad \Longleftrightarrow \quad p < \frac{2}{\alpha},  \end{displaymath} 
this shows that $L^{2 + \epsilon}$-integrability claimed in Theorem \ref{main} is sharp.

\begin{remark} The localisation in Theorem \ref{main3} is necessary: for $0 < \alpha < 1$, the weight $\mu = |x|^{-\alpha} \, dx$ has no BoA representations in the sense of Definition \ref{albertiDef}, where we require that $\tn(\Omega) < \infty$. Indeed, let $\calC = \calC(e,\theta)$ be an arbitrary cone, and assume that $(\Omega,\tn,\gamma,\tfrac{d\mu}{d\nu})$ is an Alberti representation of $\mu$ by $\calC$-graphs. Let $e' \perp e$, and let $T$ be a strip of width $1$ around $\spa(e')$. Then $\mu(T) = \infty$. However, $\calH^{1}(\gamma \cap T) \lesssim_{\theta} 1$ for all $\gamma \in \Gamma_{\calC}$, and hence $\nu(T) \lesssim_{\theta} \tn(\Omega) < \infty$. This implies that $\tfrac{d\mu}{d\nu} \notin L^{\infty}(\nu)$.
\end{remark}

\begin{notation} For $A,B > 0$, the notation $A \lesssim B$ will signify that there exists a constant $C \geq 1$ such that $A \leq CB$. This is typically used in a context where one or both of $A,B$ are functions of some variable "$x$": then $A(x) \lesssim B(x)$ means that $A(x) \leq CB(x)$ for some constant $C \geq 1$ independent of $x$. Sometimes it is worth emphasising that the constant $C$ depends on some parameter "$p$", and we will signal this by writing $A \lesssim_{p} B$.   \end{notation} 

\subsection{Higher dimensions, and connections to PDEs} The problems discussed above have natural -- but harder -- generalisations to higher dimensions. A collection of $d$ cones $\calC_{1},\ldots,\calC_{d} \subset \R^{d}$ is called \emph{independent} if $|\mathrm{det} (v_{1},\ldots,v_{d})| \geq \tau > 0$ for any choices $v_{j} \in \calC_{j} \cap S^{d - 1}$, $1 \leq j \leq d$. With this definition in mind, one can discuss Radon measures on $\R^{d}$ with $d$ independent Alberti representations. It follows from the recent breakthrough work of De Philippis and Rindler \cite{DR} that such measures are absolutely continuous with respect to Lebesgue measure. It is tempting to ask for more quantitative statements, similar to the ones in Theorem \ref{main}. Such statement do not appear to easily follow from the strategy in \cite{DR}. 
\begin{question}\label{mainQ} If $\mu$ is a Radon measure on $\R^{d}$ with $d$ independent BoA representations, then is $\mu \in L^{p}$ for some $p > 1$? 
\end{question}

In the case of independent axis-parallel representations, $\mu \in L^{d/(d - 1)}$, see the next paragraph. This is the best exponent, as can be seen by a variant of Example \ref{mainEx}. In general, we do not know how to prove $\mu \in L^{p}$ for any $p > 1$. Some results of this nature will likely follow from work in progress recently announced by Cs\"ornyei and Jones. 

Question \ref{mainQ} is closely connected with the analogue of the multilinear Kakeya problem for thin neighbourhoods of $\calC$-graphs. A near-optimal result on this variant of the multilinear Kakeya problem is contained in the paper \cite{Gu} of Guth, see \cite[Theorem 7]{Gu}. We discuss this connection explicitly in \cite[Section 5]{BKO}. It seems that the "$S^{\epsilon}$-factor" in \cite[Theorem 7]{Gu} makes it inapplicable to Question \ref{mainQ}, and it does not even imply the qualitative absolute continuity of $\mu$ established in \cite{DR}. On the other hand, the analogue of \cite[Theorem 7]{Gu} without the $S^{\epsilon}$-factor would imply a positive answer to Question \ref{mainQ} with $p = d/(d - 1)$, see the proof of \cite[Lemma 5.2]{BKO}. We do not know if this is a plausible strategy, but it certainly works for the axis-parallel case: the analogue of \cite[Theorem 7]{Gu} for neighbourhoods of axis-parallel lines is simply the classical Loomis-Whitney inequality (see \cite{LW} or \cite[Theorem 3]{Gu}), where no $S^{\epsilon}$-factor appears.

As mentioned above, the main results in this paper, and Question \ref{mainQ}, are related to the recent work of De Philippis and Rindler \cite{DR} on $\mathcal{A}$-free measures. Introducing the notation of \cite{DR} would be a long detour, but let us briefly explain some connections, assuming familiarity with the terminology of \cite{DR}.

The qualitative absolute continuity result, mentioned above Question \ref{mainQ}, follows from \cite[Corollary 1.12]{DR} after realising that, for each Alberti representation of $\mu$, 
\eqref{nuDef} may be used to construct a normal $1$-current $T_i = \vec{T_i}\|T_i\|$ on $\R^{d}$, $1 \leq i \leq d$, such that $\mu \ll \|T_i\|$. The independence of the representations translates into the statement
\begin{equation}\label{eq:DR} \dim \spa \{\vec{T}_{1}(x),\ldots,\vec{T}_{d}(x)\} = d \qquad \text{for $\mu$ a.e. } x \in \R^{d}. \end{equation}
One may view the $d$-tuple of normal currents $\mathbf{T} = (T_{1},\ldots,T_{d})$ as an $\R^{d \times d}$-valued measure $\mathbf{T}= \vec{\mathbf{T}} \|\mathbf{T}\|$, where $|\vec{\mathbf{T}}| \equiv 1$, and $\|\mathbf{T}\|$ is a finite positive measure. Since each $T_i$ is normal, $\mathrm{div}\,\mathbf{T}$ is also a finite measure, and this is the key point relating our situation with the work of De Philippis and Rindler.
If the Alberti representations of $\mu$ are BoA, then $d\mu/d\|\mathbf{T}\| \in L^{\infty}(\|\mathbf{T}\|)$, and $\mu \in L^{2}(\R^{2})$ by Theorem \ref{main}. As far as we know, PDE methods do not yield the same conclusion. However, if in addition the Jacobian of $\vec{\mathbf{T}}$ is uniformly bounded from below $\|\mathbf{T}\|$ almost everywhere, PDE methods look more promising. We formulate the following question, which is parallel to Question \ref{mainQ}:

\begin{question} Let $\mathbf{T}= \vec{\mathbf{T}} \|\mathbf{T}\|$ be a finite $\R^{d \times d}$-valued measure, whose divergence is also a finite (signed) measure such that the Jacobian of $\vec{\mathbf{T}}$ is uniformly bounded from below in absolute value $\|\mathbf{T}\|$ a.e. Is it true that $\|\mathbf{T}\| \in L^{p}(\R^d)$ for some $p > 1$? \end{question}

\subsection{Acknowledgements} We would like to thank Vesa Julin for many useful conversations on the topics of the paper. We also thank the anonymous referee for helpful suggestions leading to Question 2.

\section{Proof of the main theorem}

We prove Theorem \ref{main} in two parts, first considering representations which are only BoA, and then representations which are both BoA and BoB at the same time.

\subsection{BoA representations}\label{boundedReps}

The first part of Theorem \ref{main} easily follows from the next, more quantitative, statement: 

\begin{thm}\label{main2} Assume that $\mu$ is a Radon measure on $\R^{2}$ with two independent BoA representations $(\Omega_{1},\tn_{1},\gamma_{1},\tfrac{d\mu}{d\nu_{1}})$ and $(\Omega_{2},\tn_{2},\gamma_{2},\tfrac{d\mu}{d\nu_{2}})$. Then
\begin{equation}\label{L2bound} \|\mu\|_{2} \lesssim \prod_{j = 1}^{2} \sqrt{\tn_{j}(\Omega_{j})\|\mu\|_{L^{\infty}(\nu_{j})}}, \end{equation}
where the implicit constant only depends on the opening angles $\theta_{1},\theta_{2}$ and angular separation $\tau$ of the cones $\calC_{1} = \calC(e_{1},\theta_{1})$ and $\calC_{2} = \calC(e_{2},\theta_{2})$.
\end{thm}


\begin{proof} It suffices to show that the restriction of $\mu$ to any dyadic square $Q_{0} \subset \R^{2}$ is in $L^{2}$, with norm bounded (independently of $Q_{0}$) as in \eqref{L2bound}. For notational simplicity, we assume that $Q_{0} = [0,1)^{2}$. Let $\calD_{n} := \calD_{n}([0,1)^{2})$, $n \in \N$, be the family of dyadic sub-squares of $[0,1)^{2}$ of side-length $2^{-n}$. Fix $n \in \N$, pick $Q \in \calD_{n}$, and write
\begin{displaymath} \Omega(Q) := \{(\omega_{1},\omega_{2}) \in \Omega_{1} \times \Omega_{2} : \calH^{1}(Q \cap \gamma(\omega_{1})) > 0 \text{ and }  \calH^{1}(Q \cap \gamma(\omega_{2})) > 0\}. \end{displaymath}
Note that $\{\omega \in \Omega_{j} : \calH^{1}(\gamma(\omega) \cap Q) > 0\} \in \Sigma_{j}$ for $j \in \{1,2\}$ by \eqref{measurability}, so $\Omega(Q)$ lies in the $\sigma$-algebra generated by $\Sigma_{1} \times \Sigma_{2}$. We start by showing that
\begin{equation}\label{form4} \sum_{Q \in \calD_{n}} \chi_{\Omega(Q)}(\omega_{1},\omega_{2}) \lesssim_{\tau} 1, \qquad (\omega_{1},\omega_{2}) \in \Omega_{1} \times \Omega_{2}. \end{equation}
To prove \eqref{form4}, it suffices to fix a pair $(\gamma_{1},\gamma_{2}) \in \Gamma_{1} \times \Gamma_{2}$, where $\Gamma_{j} := \Gamma_{\calC_{j}}$, and show that there are $\lesssim_{\tau} 1$ squares $Q \in \calD_{n}$ with $\gamma_{1} \cap Q \neq \emptyset \neq \gamma_{2} \cap Q$. So, fix $(\gamma_{1},\gamma_{2}) \in \Gamma_{1} \times \Gamma_{2}$, and assume that there is at least one square $Q$ such that $\gamma_{1} \cap Q \neq \emptyset \neq \gamma_{2} \cap Q$, see Figure \ref{fig1}.
\begin{figure}[h!]
\begin{center}
\includegraphics[scale = 0.7]{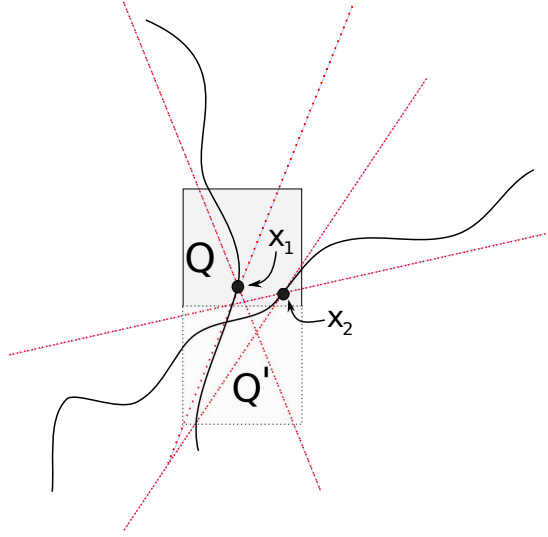}
\caption{The curves $\gamma_{1},\gamma_{2}$ and the square $Q$.}\label{fig1}
\end{center}
\end{figure}
To simplify some numerics, assume that $Q = [0,2^{-n})^{2}$. Pick $x_{1} \in \gamma_{1} \cap Q$ and $x_{2} \in \gamma_{2} \cap Q$, and note that
\begin{displaymath} \gamma_{1} \subset x_{1} + \calC_{1} \quad \text{and} \quad \gamma_{2} \subset x_{2} + \calC_{2}, \end{displaymath}
since $\gamma_{1} \in \Gamma_{\calC_{1}}$ and $\gamma_{2} \in \Gamma_{\calC_{2}}$. It follows that whenever $Q' \in \calD_{n}$ is another square with $\gamma_{1} \cap Q' \neq \emptyset \neq \gamma_{2} \cap Q'$, we can find points
\begin{displaymath} x_{1}' \in \gamma_{1} \cap Q' \subset (x_{1} + \calC_{1}) \cap Q' \quad \text{and} \quad x_{2}' \in \gamma_{2} \cap Q' \subset (x_{2} + \calC_{2}) \cap Q', \end{displaymath}
which then satisfy $\dist(x_{j}',\calC_{j}) \lesssim 2^{-n}$ for $j \in \{1,2\}$, because $|x_{j}| \lesssim 2^{-n}$. Consequently, 
\begin{equation}\label{form3} \dist(Q',\calC_{j}) \lesssim 2^{-n}, \qquad j \in \{1,2\}. \end{equation}
But the independence assumption \eqref{independence} implies that $\dist(y,\calC_{1}) \gtrsim_{\tau} |y|$ or $\dist(y,\calC_{2}) \gtrsim_{\tau} |y|$ for any $y \in \R^{2}$, and in particular the centre of $Q'$. Hence, \eqref{form3} shows that $\dist(Q',Q) \leq \dist(Q',0) \lesssim 2^{-n}$, and \eqref{form4} follows.

Now, we can finish the proof of the theorem. Given any square $Q \in \calD_{n}$, we note that $\calH^{1}(Q \cap \gamma) \lesssim_{\theta_{1},\theta_{2}} 2^{-n}$ for all $\gamma \in \Gamma_{1} \cup \Gamma_{2}$ , whence
\begin{align} \mu(Q) & \leq \|\mu\|_{L^{\infty}(\nu_{j})}\int_{\Omega} \calH^{1}(Q \cap \gamma(\omega)) \, d\tn_{j}(\omega) \notag\\
&\label{form5} \lesssim_{\theta_{1},\theta_{2}} \|\mu\|_{L^{\infty}(\nu_{j})}\tn_{j}(\{\omega \in \Omega_{j} : \calH^{1}(Q \cap \gamma(\omega)) > 0\}) \cdot 2^{-n}, \quad j \in \{1,2\}. \end{align}
Observe that
\begin{displaymath} \tn_{1}(\{\omega \in \Omega_{1} : \nu_{\omega}^{1}(Q) > 0\})\tn_{2}(\{\omega \in \Omega_{2} : \nu_{\omega}^{2}(Q) > 0\}) = (\tn_{1} \times \tn_{2})(\Omega(Q)). \end{displaymath}
Denoting the Lebesgue measure of $Q$ by $|Q|$, and combining \eqref{form5} with \eqref{form4} gives
\begin{align*} \sum_{Q \in \calD_{n}} \left(\frac{|\mu(Q)|}{|Q|} \right)^{2}|Q| & \lesssim_{\theta_{1},\theta_{2}} \|\mu\|_{L^{\infty}(\nu_{1})}\|\mu\|_{L^{\infty}(\nu_{2})} \sum_{Q \in \calD_{n}} (\tn_{1} \times \tn_{2})(\Omega(Q))\\
& = \|\mu\|_{L^{\infty}(\nu_{1})}\|\mu\|_{L^{\infty}(\nu_{2})}\int \sum_{Q \in \calD_{n}} \chi_{\Omega(Q)} \, d(\tn_{1} \times \tn_{2})\\
& \lesssim_{\tau} \|\mu\|_{L^{\infty}(\nu_{1})}\|\mu\|_{L^{\infty}(\nu_{2})} \tn_{1}(\Omega_{1})\tn_{2}(\Omega_{2}).  \notag \end{align*}
This inequality shows that the $L^{2}$-norms of the measures
\begin{displaymath} \mu_{n} := \sum_{Q \in \calD_{n}} \frac{\mu(Q)}{|Q|}\chi_{Q}, \qquad n \in \N, \end{displaymath}
are uniformly bounded by the right hand side of \eqref{L2bound}. The proof can then be completed by standard weak convergence arguments. \end{proof} 

\subsection{Representations which are both BoA and BoB} Before finishing the proof of Theorem \ref{main}, we need to record a few geometric observations.

\begin{lemma} Let $\mathbf{n} := (0,1)$, $\tau > 0$, and let $v = (e_{1},e_{2}) \in S^{1}$ with $e_{2} \leq -\tau < 0$. Then, the following holds for $\epsilon := \min\{\tau^{4},10^{-4}\}$:
\begin{displaymath} B(\mathbf{n},\epsilon) + tv \subset B(0,1), \qquad t \in [\sqrt{\epsilon},2\sqrt{\epsilon}]. \end{displaymath}
\end{lemma}

\begin{proof} Write $\kappa := \min\{\tau,10^{-1}\}$, so that $\epsilon = \kappa^{4}$ and $\kappa^{2} = \sqrt{\epsilon}$. Then, fix $x \in B(\mathbf{n},\epsilon)$ and $t \in (\kappa^{2},2\kappa^{2})$. Noting that $\mathbf{n} \cdot v = e_{2}$, $|x|^{2} \leq 1 + 3\epsilon$, and $|x - \mathbf{n}| \leq \epsilon$, we compute that
\begin{align*} |x + tv|^{2} & = |x|^{2} + 2x \cdot tv + t^{2}\\
& \leq 1 + 3\epsilon + 2(x - \mathbf{n}) \cdot tv + 2e_{2}t + t^{2}\\
& \leq 1 + 5\epsilon + 2e_{2}t + t^{2} < 1,  \end{align*}
because (using first that $e_{2} < 0$ and $\kappa^{2} < t < 2\kappa^{2}$, and then that $e_{2} \leq -\tau$ and $\kappa = \min\{\tau,10^{-1}\}$)
\begin{displaymath} 5\epsilon + 2e_{2}t + t^{2} < 5\kappa^{4} + 2e_{2}\kappa^{2} + 4\kappa^{4} = \kappa^{2}(9\kappa^{2} + 2e_{2}) < 0. \end{displaymath} 
This completes the proof. \end{proof}

In the next corollary, we write 
\begin{displaymath} A(x,r,R) := B(x,R) \, \setminus \, B(x,r) \end{displaymath}
for $x \in \R^{2}$ and $0 < r < R < \infty$. Also, if $\calC = \calC(e,\theta) \subset \R^{2}$ is a cone, we write
\begin{displaymath} \calC^{+} := \{x \in \R^{2} :  x \cdot e \geq \theta |x|\} \quad \text{and} \quad \calC^{-} := \{x \in \R^{2} : x \cdot e \leq -\theta |x|\} \end{displaymath}
for the corresponding "one-sided" cones.

\begin{cor}\label{coneCor} Let $\calC_{1},\calC_{2} \subset \R^{2}$ be two cones with
\begin{displaymath} \min\{\angle(x_{1},x_{2}) : x_{1} \in \calC_{1} \, \setminus \, \{0\} \text{ and } x_{2} \in \calC_{2} \, \setminus \, \{0\}\} \geq \tau > 0. \end{displaymath}
Then, the following holds for $\epsilon := \min\{|\tau/100|^{4},10^{-4}\}$, and for any $x \in \R^{2}$, $r > 0$, and $n \in \partial B(x,r)$. There exists $j \in \{1,2\}$ and a sign $\star \in \{-,+\}$ (depending only on $x,n$) such that
\begin{equation}\label{form17} A(y,r\sqrt{\epsilon},2r\sqrt{\epsilon}) \cap [y + \calC_{j}^{\star}]  \subset B(x,r), \qquad y \in B(n,\epsilon r). \end{equation} \end{cor}
The statement is best illustrated by a picture, see Figure \ref{fig3}.
\begin{figure}[h!]
\begin{center}
\includegraphics[scale = 0.6]{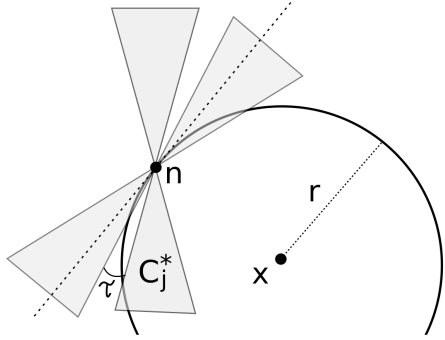}
\caption{The scenario in Corollary \ref{coneCor}. One of the four half-cones always has large intersection with $B(x,r)$.}\label{fig3}
\end{center}
\end{figure}

\begin{proof}[Proof of Corollary \ref{coneCor}] After rescaling, translation, and rotation, we may assume that 
\begin{equation}\label{form18} x = 0, \quad r = 1, \quad \text{and} \quad n = \mathbf{n} = (0,1). \end{equation}
Write $\pi_{y}(x,y) := y$. We start by noting that
\begin{equation}\label{form12} \pi_{y}(\calC_{j} \cap S^{1}) \cap [-\tfrac{\tau}{100},\tfrac{\tau}{100}] = \emptyset \end{equation}
for either $j = 1$ or $j = 2$. If this were not the case, we could find $x_{1} \in \calC_{1} \cap S^{1}$ and $x_{2} \in \calC_{2} \cap S^{1}$ such that $|\mathrm{sin} \, \angle (x_{j},(1,0))| = |\pi_{y}(x_{j})| \leq \tau/100$ for $j \in \{1,2\}$. Then either $\angle (x_{1},x_{2}) < \tau$ or $\angle (x_{1},-x_{2}) < \tau$. Both contradict the definition of $\tau$, given that also $x_{2} \in \calC_{2}$. This proves \eqref{form12}.

Fix $j \in \{1,2\}$ such that \eqref{form12} holds, and write, for $\star \in\{-,+\}$,
\begin{displaymath} \calC_{j}^{\star} \cap S^{1} =: J_{j}^{\star}. \end{displaymath}
Then $J_{j}^{+} = -J_{j}^{-}$, and consequently $\pi_{y}(J_{j}^{+}) = -\pi_{y}(J_{j}^{-})$. It follows from this, \eqref{form12}, and the fact that $\pi_{y}(J_{j}^{\star})$ is an interval, that either $\pi_{y}(v) < -\tau/100$ for all $v \in J_{j}^{+}$ or $\pi_{y}(v) < -\tau/100$ for all $v \in J_{j}^{-}$. We pick $\star \in \{-,+\}$ such that this conclusion holds. In other words, the $y$-coordinate of every point $v \in \mathcal{C}^{\star}_{j} \cap S^{1}$ is $< -\tau/100$. It follows from the previous lemma, and the choice of $\epsilon$, that
\begin{displaymath} B(\mathbf{n},\epsilon) + tv \subset B(0,1), \qquad v \in \mathcal{C}_{j}^{\star} \cap S^{1}, \: t \in [\sqrt{\epsilon},2\sqrt{\epsilon}], \end{displaymath}
which is equivalent to \eqref{form17} (recalling \eqref{form18}). \end{proof}

For the rest of the section, we assume that $\mu$ is a Radon measure on $\R^{2}$ with $\mu(B(1)) > 0$, and that $\mu$ has two independent Alberti representations which are both BoA and BoB on $B(2)$. Thus, there exists a constant $C \geq 1$ such that
\begin{equation}\label{controlledDensities} C^{-1}\nu(A) \leq \mu(A) \leq C\nu(A) \end{equation}
for all Borel sets $A \subset B(2)$. By Theorem \ref{main2}, we already know that $\mu \in L^{2}(B(1))$. We next aim to show that $B(1) \subset \spt \mu$, and $\mu$ is a doubling weight on $B(1)$ in the following sense:
\begin{equation}\label{doubling} B(x,r) \subset B(1) \quad \Longrightarrow \quad \mu(B(x,\tfrac{3}{2}r)) \lesssim_{C,\tau} \mu(B(x,r)). \end{equation}
After this, it will be easy to complete the proof of the reverse H\"older inequality \eqref{reverseHolder}.

\begin{lemma}\label{doublingLemma} Let $\mu$ be a measure as above. Then $\mu$ is doubling on $B(1)$ in the sense of \eqref{doubling}, where the constants only depend on $C$ from \eqref{controlledDensities} and $\tau$ from \eqref{independence}. In particular, $B(1) \subset \spt \mu$. \end{lemma}

\begin{remark} To prove the "in particular" statement, recall that $\mu(B(1)) > 0$, so $B(1) \cap \spt \mu \neq \emptyset$. Hence, if $\spt \mu \subsetneq B(1)$, one could find a ball $B(x,r) \subset B(1)$ such that $\mu(B(x,r)) = 0$, but $\partial B(x,r) \cap \spt \mu \neq \emptyset$. This would immediately violate \eqref{doubling}.  \end{remark}

\begin{proof}[Proof of Lemma \ref{doublingLemma}] Let $0 < \epsilon < 1/10$ be the parameter given by Corollary \ref{coneCor}, applied with the angular separation constant $\tau > 0$ of the cones $\calC_{1},\calC_{2}$. It suffices to argue that
\begin{displaymath} B(x,r) \subset B(\tfrac{3}{2}) \quad \Longrightarrow \quad \mu(B(x,(1 + \tfrac{\epsilon}{2})r) \, \setminus \, B(x,r)) \lesssim_{C,\epsilon} \mu(B(x,r)). \end{displaymath} 
Cover the annulus $B(x,(1 + \tfrac{\epsilon}{2})r) \, \setminus \, B(x,r)$ by a minimal number of balls $B_{1},\ldots,B_{N}$ of radius $\epsilon r$ centred on $\partial B(x,r)$, and let $B := B_{i} = B(x_{i},\epsilon r)$ be the ball maximising $B_{i} \mapsto \mu(B_{i})$. Since $N \lesssim 1/\epsilon$, we have
\begin{displaymath} \mu(B) \gtrsim \epsilon \mu(B(x,(1 + \tfrac{\epsilon}{2})r) \, \setminus \, B(x,r)), \end{displaymath}
and consequently it suffices to show that $\mu(B(x,r)) \gtrsim_{C,\epsilon} \mu(B)$. Recalling \eqref{controlledDensities}, and noting that $B(x,r) \cup B \subset B(2)$, this will follow once we manage to show that
\begin{equation}\label{form32} \nu_{j}(B(x,r)) \gtrsim_{\epsilon} \nu_{j}(B) \end{equation}
for either $j = 1$ or $j = 2$.

For $y \in \R^{2}$, write $A(y)$ for the annulus
\begin{displaymath} A(y) := A(y,r\sqrt{\epsilon},2r\sqrt{\epsilon}). \end{displaymath} 
Recall the half-cones $\calC_{j}^{\star}$, $\star \in \{-,+\}$, defined above Corollary \ref{coneCor}. By Corollary \ref{coneCor}, there exist choices of $j \in \{1,2\}$ and $\star \in \{-,+\}$, depending only on $x$ and $x_{i} \in \partial B(x,r)$ (i.e. the centre of $B$), such that
\begin{displaymath} y \in B \quad \Longrightarrow \quad  A(y) \cap [y + \mathcal{C}_{j}^{\star}] \subset B(x,r). \end{displaymath} 
Consequently,
\begin{displaymath} G := \bigcup_{y \in B} (A(y) \cap [y + \calC_{j}^{\star}]) \subset B(x,r). \end{displaymath}
Define
\begin{displaymath} \Omega_{j}(B) := \{\omega \in \Omega_{j} : \calH^{1}(B \cap \gamma_{j}(\omega)) > 0\} \in \Sigma_{j}. \end{displaymath} 
We observe that if $\omega \in \Omega_{j}(B)$, then $\calH^{1}(G \cap \gamma_{j}(\omega)) \sim_{\epsilon} r$. Indeed, if $\omega \in \Omega_{j}(B)$, then certainly $\gamma(\omega)$ contains a point $y \in B$ and then one half of the graph $\gamma_{j}(\omega)$ is contained in $y + \calC_{j}^{\star}$. This half intersects $A(y)$ in length $\sim_{\epsilon} r$, and the intersection is contained in $G$ by definition. It follows that
\begin{align*} \nu_{j}(G) & \geq \int_{\Omega_{j}(B)} \calH^{1}(G \cap \gamma_{j}(\omega)) \, d\tn_{j}(\omega) \gtrsim_{\epsilon} r \cdot \tn_{j}(\Omega_{j}(B))\\
& \gtrsim \int_{\Omega_{j}(B)} \calH^{1}(B \cap \gamma_{j}(\omega)) \, d\tn_{j}(\omega) = \nu_{j}(B).  \end{align*}
Since $G \subset B(x,r)$, this yields \eqref{form32} and completes the proof. \end{proof}

We can now complete the proof of the reverse H\"older inequality \eqref{reverseHolder}.
\begin{proof}[Concluding the proof of Theorem \ref{main}] Fix a ball $B := B(x,r) \subset B(1)$, and consider the restrictions of the measures $\tn_{1},\tn_{2}$ to the sets
\begin{displaymath} \Omega_{j}(B) := \{\omega \in \Omega_{j} : \calH^{1}(B \cap \gamma_{j}(\omega)) > 0\} \in \Sigma_{j}, \qquad j \in \{1,2\}. \end{displaymath}
Writing $\tn_{j}^{B} := (\tn_{j})|_{\Omega_{j}(B)}$, the restriction $\mu^{B} := \mu|_{B}$ has two independent Alberti representations $\{\Omega_{j}(B),\tn_{j}^{B},\gamma_{j},\tfrac{d\mu^{B}}{d\nu_{j}}\}$, $j \in \{1,2\}$. Evidently $\|\mu^{B}\|_{L^{\infty}(\nu_{j})} \leq C$ for $j \in \{1,2\}$, where $C \geq 1$ is the constant from \eqref{controlledDensities}, so we may deduce from Theorem \ref{main2} that
\begin{displaymath} \|\mu^{B}\|_{2} \lesssim_{\tau,\theta_{1},\theta_{2}} C \sqrt{\tn_{1}(\Omega_{1}(B))\tn_{2}(\Omega_{2}(B))}. \end{displaymath} 
It remains to prove that
\begin{equation}\label{form16} r \cdot \sqrt{\tn_{1}(\Omega_{1}(B))\tn_{2}(\Omega_{2}(B))} \lesssim_{C,\tau} \mu(B(x,r)), \end{equation} 
since the reverse H\"older inequality \eqref{reverseHolder} is equivalent to $\|\mu\|_{L^{2}(B(x,r))} \lesssim_{C,\tau} r^{-1} \cdot \mu(B(x,r))$. To see this, we note that
\begin{displaymath} r \cdot \tn_{j}(\Omega_{j}(B)) \lesssim \int_{\Omega_{j}(B)} \calH^{1}(B(x,\tfrac{3}{2}r) \cap \gamma_{j}(\omega)) \, \tn_{j}(\omega) = \nu_{j}(B(x,\tfrac{3}{2}r)) \leq C\mu(B(x,\tfrac{3}{2}r)) \end{displaymath} 
for $j \in \{1,2\}$, because any $\gamma \in \Gamma_{\calC_{j}}$ meeting $B$ satisfies $\calH^{1}(B(x,\tfrac{3}{2}r) \cap \gamma) \sim r$. Taking a geometric average over $j \in \{1,2\}$, this implies \eqref{form16} with $\mu(B(x,\tfrac{3}{2}r))$ on the right hand side. But since $B(x,r) \subset B(1)$, Lemma \ref{doublingLemma} yields $\mu(B(x,\tfrac{3}{2}r)) \lesssim_{C,\tau} \mu(B(x,r))$. This completes the proofs of \eqref{form16} and Theorem \ref{main}. \end{proof}

\section{Sharpness of the reverse H\"older exponent}

The purpose of this section is to prove Theorem \ref{main3}. The statement is repeated below:
\begin{thm}\label{exThm} Let $0 < \alpha < 1$. The measure $\mu = f\, d\calL^{2}$, where
\begin{equation}\label{form23} f(x) = |x|^{-\alpha}\mathbf{1}_{[-1,1]^{2}}, \end{equation}
has two independent Alberti representations which are both BoA and BoB on $[-1,1]^{2}$. \end{thm}

\begin{remark} It may be worth pointing out that, in the construction below, the BoA and BoB constants stay uniformly bounded for $\alpha \in (0,1)$. However, the independence constant of the two representations (that is, the constant "$\tau$" from \eqref{independence}) tends to zero as $\alpha \nearrow 1$. In this section, the constants hidden in the "$\sim$" and "$\lesssim$" notation will not depend on $\alpha$. \end{remark}

We have replaced $B(1)$ by $[-1,1]^{2}$ for technical convenience; since $B(1) \subset [-1,1]^{2}$, the result is technically stronger than Theorem \ref{main3}. The two representations will be denoted by $\{\Omega_{1},\tn_{1},\gamma_{1},\tfrac{d\mu}{d\nu_{1}}\}$ and $\{\Omega_{2},\tn_{2},\gamma_{2},\tfrac{d\mu}{d\nu_{2}}\}$. We will first construct one representation of $\mu$ restricted to $[0,1]^{2}$, as in Figure \ref{fig2}, and eventually extend that representation to $[-1,1]^{2}$, as on the left hand side of Figure \ref{fig4}. We set
\begin{displaymath} \Omega_{1} := [-1,1] \times \{1\} =: \Omega_{2}, \end{displaymath}
and we let $\tn=\tn_{j} := \calH^{1}|_{\Omega_{j}}$. The main challenge is of course to construct the graphs $\gamma_{j}(\omega)$, $\omega \in \Omega_{j}$. A key feature of $f$ is that $f(r,t) = f(t,r)$ for $(r,t) \in [-1,1]^{2}$. Hence, as we will argue carefully later, it suffices to construct a single representation by $\calC$-graphs, where $\calC$ is a cone around the $y$-axis, with opening angle strictly smaller than $\pi/2$; such a representation is depicted on the left hand side of Figure \ref{fig2}. We remark that, as the picture suggests, every $\calC$-graph associated to the representation can be expressed as a countable union of line segments. The second representation is eventually acquired by rotating the first representation by $\pi/2$, see the right hand side of Figure \ref{fig4}. 
\begin{figure}[h!]
\begin{center}
\includegraphics[scale = 0.6]{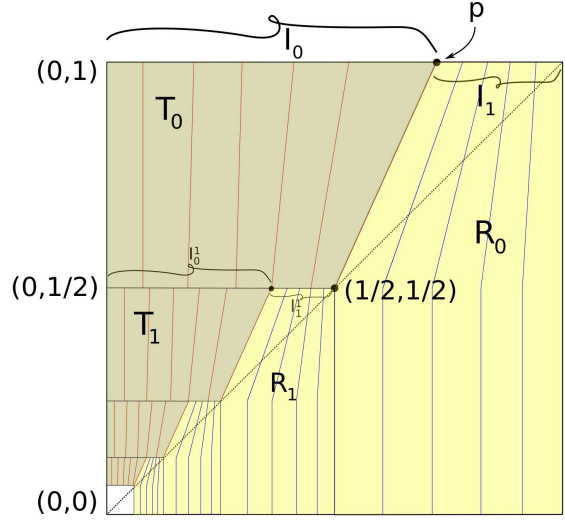}
\caption{One representation of $\mu|_{[0,1]^{2}}$.}\label{fig2}
\end{center}
\end{figure}

Now we construct certain graphs $\gamma(\omega)$ for $\omega \in \Omega := [0,1] \times \{1\} \subset \Omega_{1}$. The idea is that eventually $\gamma_{1}(\omega) \cap [0,1]^{2} = \gamma(\omega)$ for $\omega \in \Omega$. The graphs $\gamma(\omega)$ will be constructed so that
\begin{equation}\label{form10} \gamma(\omega) \cap \Omega = \{\omega\}, \qquad \omega \in \Omega. \end{equation}
The right idea to keep in mind is that the graph $\gamma(\omega)$ "starts from $\omega \in \Omega = [0,1] \times \{1\}$, travels downwards, and ends somewhere on $[0,1] \times \{0\}$". We will ensure that $[0,1]^{2}$ is foliated by the graphs $\gamma(\omega)$, $\omega \in \Omega$. 

Start by fixing a point $p \in \Omega$ whose $x$-coordinate lies in $(1/2,1)$, see Figure \ref{fig2}. The relationship between $p$ and the exponent $\alpha$ in \eqref{form23} will be specified under \eqref{betaChoice}. Let
\begin{displaymath} I_{0} := [(0,1),p] \subset \Omega \quad \text{and} \quad I_{1} := (p,(1,1)] \subset \Omega. \end{displaymath}
We can now specify the graphs $\gamma(\omega)$ with $\omega \in I_{1}$. Each of them consists of two line segments: the first one connects $I_{1}$ to $((\tfrac{1}{2},\tfrac{1}{2}),(1,\tfrac{1}{2})]$, and the second one is vertical, connecting $((\tfrac{1}{2},\tfrac{1}{2}),(1,\tfrac{1}{2})]$ to $[1/2,1] \times \{0\}$, see Figure \ref{fig2}. We also require that the graphs $\gamma(\omega)$ foliate the yellow pentagon $R_{0}$ in Figure \ref{fig2}. This description still gives some freedom on how to choose the first segments, but if the choice is done in a natural way, we will find that
\begin{equation}\label{form11} |\{\omega \in I_{1} : \gamma(\omega) \cap B \neq \emptyset\}| \sim \diam(B) \end{equation}
for all balls $B \subset R_{0}$. Here, and in the sequel, $|\cdot|$ stands for $1$-dimensional Hausdorff measure. The implicit constant of course depends on the length of $I_{1}$ (and hence $p$, and eventually $\alpha$). 

We then move our attention to defining the graphs $\gamma(\omega)$ with $\omega \in I_{0}$.  Look again at Figure \ref{fig2} and note the green trapezoidal regions, denoted by $T_{j}$, $j \geq 0$. To be precise, $T_{0}$ is the convex hull of $I_{0} \cup [(0,\tfrac{1}{2}),(\tfrac{1}{2},\tfrac{1}{2})]$, and 
\begin{displaymath} T_{j} := 2^{-j} T_{0} = \{2^{-j}(x,y) : (x,y) \in T_{1}\}, \qquad j \geq 1. \end{displaymath}
For $j \geq 0$, we also define
\begin{displaymath} \Omega^{j} := 2^{-j}\Omega = [(0,2^{-j}),(2^{-j},2^{-j})], \quad I_{0}^{j} := 2^{-j}I_{0}, \quad \text{and} \quad I_{1}^{j} := 2^{-j}I_{1}. \end{displaymath}
Then $\Omega^{j}$ is the bottom edge of the trapezoid $T_{j - 1}$, and $I_{0}^{j}$ is the top edge of the trapezoid $T_{j}$ for $j \geq 1$. Also,
\begin{displaymath} \Omega^{j} = I_{0}^{j} \cup I_{1}^{j}, \qquad j \geq 0. \end{displaymath}
We point out that $\Omega^{0} = \Omega$, $I_{0}^{0} = I_{0}$ and $I_{1}^{0} = I_{1}$. 

We then construct initial segments of the graphs $\gamma(\omega)$, $\omega \in I_{0}$ as follows. Define the map $\sigma_{0} \colon I_{0} \to \Omega^{1}$ by 
\begin{displaymath} \sigma_{0}(x,1) := (\beta x,\tfrac{1}{2}), \end{displaymath}
where $\beta = \beta(p) \in (\tfrac{1}{2},1)$ is chosen so that $\sigma_{0}(I_{0}) = \Omega^{1}$, that is, 
\begin{equation}\label{betaChoice} \beta(p) := \frac{1}{2 \cdot |(0,1) - p|}. \end{equation}
Note that as $p$ varies in $((\tfrac{1}{2},1),(1,1))$, the number $\beta(p)$ takes all values in $(\tfrac{1}{2},1)$. In particular, we may choose $\beta(p) = 2^{-\alpha}$, where $\alpha \in (0,1)$ is the exponent in \eqref{form23}.

Now, we connect every $\omega \in I_{0}$ to $\sigma_{0}(\omega)$ by a line segment, see Figure \ref{fig2}; this is an initial segment of $\gamma(\omega)$. We record that if $I \subset \Omega^{1}$ is any horizontal segment (or even a Borel set), then
\begin{equation}\label{form20}  \tn\{\omega \in I_{0} : \gamma(\omega) \cap I \neq \emptyset\} = \tn\{\omega \in I_{0} : \sigma_{0}(\omega) \in I \} = \beta^{-1} \cdot |I|. \end{equation}

Now, we have defined the intersections of the curves $\gamma(\omega)$ with $T_{0} \cup R_{0}$. In particular, the following set families are well-defined:
\begin{displaymath} \Gamma(T_{0}) := \{\gamma(\omega) \cap T_{0} : \omega \in I_{0}\} \quad \text{and} \quad \Gamma(R_{0}) := \{\gamma(\omega) \cap R_{0} : \omega \in I_{1}\}. \end{displaymath}
The graphs in $\Gamma(R_{0})$ are already \emph{complete} in the sense that they connect $\Omega$ to $[0,1] \times \{0\}$. The graphs in $\Gamma(T_{0})$ are evidently not complete, and they need to be extended. To do this, we define $R_{j} := 2^{-j}R_{0}$ for $j \geq 1$, see Figure \ref{fig2}, and we define the set families
\begin{displaymath} \quad \Gamma(R_{j}) = 2^{-j}\Gamma(R_{0}), \qquad j \geq 1. \end{displaymath}
for $j \geq 1$. In other words, the sets in $\Gamma(R_{j})$ are obtained by rescaling the graphs in $\Gamma(R_{0})$ so they fit inside, and foliate, $R_{j}$. We note that the sets in $\Gamma(R_{j})$ connect points in $I_{1}^{j}$ to $[0,1] \times \{0\}$ for $j \geq 0$.

Finally, we define the complete graphs $\gamma(\omega)$, $\omega \in I_{0}$ as follows. Fix $\omega \in I_{0}$, and note that $\gamma_{0} := \gamma(\omega) \cap T_{0}$ has already been defined, and the intersection $\gamma_{0} \cap \Omega^{1}$ contains a single point $z = \sigma_{0}(\omega)$, which lies in either $I_{0}^{1}$ or $I_{1}^{1}$. If $z \in I_{1}^{1} \subset R_{1}$, then there is a unique set $\gamma_{1} \in \Gamma(R_{1})$ with $z \in \gamma_{1}$. Then we define
\begin{displaymath} \gamma(\omega) \cap [T_{0} \cup R_{1}] := \gamma_{0} \cup \gamma_{1}. \end{displaymath} 
In this case $\gamma(\omega)$ is now a complete graph, and the construction of $\gamma(\omega)$ terminates. Before proceeding with the case $z \in I_{0}^{1}$, we pause for a moment to record a useful observation. If $B \subset R_{1}$ is a ball, consider
\[\Omega^{1}(B) := \{x \in \Omega^{1} : x \in \gamma \text{ and } \gamma \cap B \neq \emptyset \text{ for some } \gamma \in \Gamma\}.\]
Since all the graphs $\gamma \in \Gamma$ entering $R_{1}$ can be written as $\gamma_{0} \cup \gamma_{1}$ with $\gamma_{0}$ terminating at $I_{1}^{1}$ and $\gamma_{1} \in \Gamma(R_{1})$, the set $\Omega^{1}(B)$ can be rewritten as
\begin{displaymath} \Omega^{1}(B) = \{x \in I_{1}^{1} : x \in \gamma \text{ and } \gamma \cap B \neq \emptyset \text{ for some } \gamma \in \Gamma(R_{1})\}. \end{displaymath}
Then, recalling that $I_{1}^{1} = 2^{-1}I_{1}$, and $\Gamma(R_{1}) = 2^{-1}\Gamma(R_{0})$, and noting that $2B \subset R_{0}$, we see that
\begin{equation}\label{form27} |\Omega^{1}(B)| = \tfrac{1}{2}|\{\omega \in I_{1} : \gamma(\omega) \cap 2B \neq \emptyset\} \sim \tfrac{1}{2} \cdot \diam(2B) = \diam(B), \end{equation}
using \eqref{form11}. The main point here is that the implicit constant is the same (absolute constant) as in \eqref{form11}. We remark that here $2B = \{2x : x \in B\}$ is the honest dilation of $B$ (and not a ball with the same centre and twice the radius as $B$).

We then consider the case $z = \sigma_{0}(\omega) \in I_{0}^{1} \subset T_{1}$. We define a map $\sigma_{1} \colon I_{0}^{1} \to \Omega^{2}$ by
\begin{displaymath} \sigma_{1}(x,\tfrac{1}{2}) := (\beta x,\tfrac{1}{4}), \end{displaymath}
and then connect every point $(x,\tfrac{1}{2}) \in I_{0}^{1}$ to $\sigma_{1}(x,\tfrac{1}{2}) \in \Omega^{2}$ by a line segment. In particular, this gives us the definition of $\gamma(\omega)$ in $T_{0} \cup T_{1}$: namely, $\gamma(\omega) \cap [T_{0} \cup T_{1}]$ is a union of two line segments, the first connecting $\omega$ to $\sigma_{0}(\omega) = z$, and the second connecting $z$ to $\sigma_{1}(z)$. We note that if $\omega = (x,1) \in I_{0}$, then $\gamma(\omega) \cap \Omega^{2}$ consists of the point $\sigma_{1}(\sigma_{0}(\omega)) = (\beta^{2}x,\tfrac{1}{4})$. For any Borel set $I \subset \Omega^{2}$, this gives 
\begin{equation}\label{form21} \tn\{\omega \in I_{0} : \gamma(\omega) \cap I \neq \emptyset\} = \beta^{-2} \cdot |I|,  \end{equation}
which is an analogue of \eqref{form20} for subsets of $\Omega^{2}$.

It is now clear how to proceed inductively, assuming that $\gamma(\omega) \cap [T_{0} \cup \ldots \cup T_{k}]$ has already been defined for some $k \geq 1$, and then considering separately the cases 
\begin{displaymath} \gamma(\omega) \cap \Omega^{k + 1} \subset I_{0}^{k + 1} \subset T_{k + 1} \quad \text{and} \quad \gamma(\omega) \cap \Omega^{k + 1} \subset I_{1}^{k + 1} \subset R_{k + 1}. \end{displaymath} 
In the case $\gamma(\omega) \cap \Omega^{k + 1} \subset I_{1}^{k + 1}$, we extend $\gamma(\omega)$ to a complete graph contained in $T_{0} \cup \ldots \cup T_{k} \cup R_{k + 1}$ by concatenating $\gamma(\omega) \cap [T_{1} \cup \ldots \cup T_{k}]$ with a set from $\Gamma(R_{k + 1})$. Arguing as in \eqref{form27}, we have in this case the following estimate for all balls $B \subset R_{k + 1}$:
\begin{equation}\label{form28} |\{x \in \Omega^{k + 1} : x \in \gamma \text{ and } \gamma \cap B \neq \emptyset \text{ for some } \gamma \in \Gamma\}| \sim \diam(B), \end{equation}
where the implicit constant is the same as in \eqref{form11}. Indeed, the set on the left hand side of \eqref{form28} is equal to a translate of $2^{-(k + 1)}\{\omega \in I_{1} : \gamma(\omega) \cap 2^{k + 1}B \neq \emptyset\}$.

In the case $\gamma(\omega) \cap \Omega^{k + 1} \subset I_{0}^{k + 1}$, we define the map $\sigma_{k + 1} \colon I_{0}^{k + 1} \to \Omega^{k + 2}$ as before:
\begin{displaymath} \sigma_{k + 1}(x,2^{-(k + 1)}) := (\beta x,2^{-(k + 2)}), \end{displaymath}
and connect the points $z \in I_{0}^{k + 1}$ to $\sigma_{k + 1}(z) \in \Omega^{k + 2}$ by line segments. Arguing as in \eqref{form20} and \eqref{form21}, we find that
\begin{equation}\label{form22} \tn\{\omega \in \Omega : \gamma(\omega) \cap I \neq \emptyset\} = \beta^{-k} \cdot |I|, \qquad k \geq 0, \: I \subset \Omega^{k} \text{ Borel}. \end{equation} 
This completes the definition of the graphs in $\Gamma$. It is easy to check inductively that graphs in $\Gamma$ foliate $(0,1]^{2}$. Moreover, the (partially defined) graph $\gamma(0,1)$ never leaves $\{0\} \times [0,1]$ during the construction, so we can simply agree that $(0,0)$ is the endpoint of $\gamma(0,1)$, thus completing the foliation of $[0,1]^{2}$.

\begin{figure}[h!]
\begin{center}
\includegraphics[scale = 0.2]{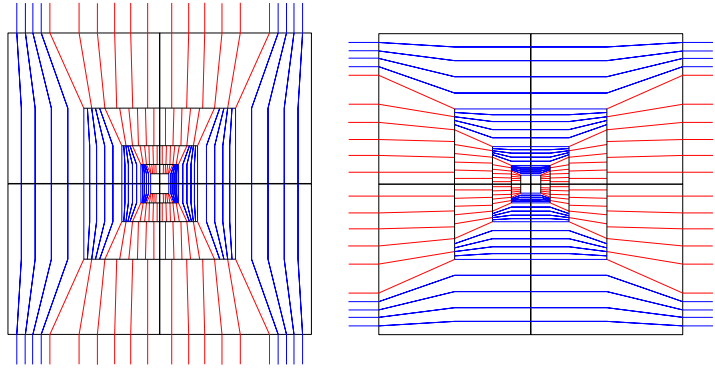}
\caption{On the left: extending the foliation of $[0,1]^{2}$ to a foliation of $[-1,1]^{2}$. On the right: the second representation.}\label{fig4}
\end{center}
\end{figure}

The sets in $\Gamma$ are clearly (non-maximal) $\calC$-graphs with respect to some cone of the form $\calC = \calC((0,1),\theta)$. As long as $p \neq (1,1)$, the opening angle of $\calC$ is strictly smaller than $\pi/2$, or in other words $\theta > \sin(\tfrac{\pi}{4})$. We then extend the graphs $\gamma(\omega) \in \Gamma$ to maximal $\calC$-graphs $\gamma_{1}(\omega)$, $\omega \in \Omega$, as follows (see Figure \ref{fig4} for an idea of what is happening). For $\omega \in \Omega \in [0,1] \times \{1\}$, let $\gamma(\omega) \subset [0,1]^{2}$ be the graph constructed above, and let 
\begin{displaymath} X(x,y) := (x,-y) \quad \text{and} \quad Y(x,y) := (-x,y) \end{displaymath}
be the reflections over the $x$-axis and $y$-axis, respectively. First concatenate $\gamma(\omega)$ with a vertical half-line starting from $\omega$ and travelling upwards. Denoting this "half-maximal" graph by $\tilde{\gamma}(\omega)$, we let 
\begin{displaymath} \gamma_{1}(\omega) := \tilde{\gamma}(\omega) \cup X(\tilde{\gamma}(\omega)), \qquad \omega \in \Omega. \end{displaymath}
Noting that $\gamma(\omega)$ has one endpoint on $[0,1] \times \{0\}$, this procedure defines a maximal $\calC$-graph $\gamma_{1}(\omega)$. Finally, recalling that $\Omega$ was only the right half of $\Omega_{1} = [-1,1] \times \{1\}$, we define
\begin{displaymath} \gamma_{1}(\omega) := \gamma_{1}(Y(\omega)), \qquad \omega \in [-1,0) \times \{1\}. \end{displaymath}
This completes the definition of the triple $(\Omega_{1},\tn_{1},\gamma_{1})$. We then consider the measure
\begin{displaymath} \nu_{1} = \nu_{(\Omega_{1},\tn_{1},\gamma_{1})} = \int_{\Omega_{1}}  \calH^{1}|_{\gamma_{1}(\omega)} \, d\tn_{1}(\omega). \end{displaymath}
Recall the measure $\mu = f \, d\calL^{2}|_{[-1,1]^{2}}$ defined in \eqref{form23}. We will next show that 
\begin{equation}\label{form15} \mu \in L^{\infty}(\nu_{1}) \quad \text{and} \quad \nu_{1} \in L^{\infty}([-1,1]^{2},\mu). \end{equation}
In other words, the Alberti representation $(\Omega,\tn,\gamma,\tfrac{d\mu}{d\nu_{1}})$ of $\mu$ by $\calC$-graphs is both BoA and BoB on $[-1,1]^{2}$. Noting that $f \circ X = f \circ Y = f$, and $X(\nu_{1}) = Y(\nu_{1}) = \nu_{1}$, it suffices to compare $\mu$ and $\nu_{1}$ on $[0,1]^{2}$. Moreover, it suffices to show that the Radon-Nikodym derivative $(d\nu_{1}/d\calL^{2})(z)$ at $\calL^{2}$ almost every interior point $z$ of one of the regions $T_{k}$ or $R_{k}$ is comparable to $f(z)$. 

Assume first that $k \geq 0$ and $z \in \operatorname{int} T_{k}$, and fix $r > 0$ so small that $B := B(z,r) \subset T_{k}$. Then
\begin{equation}\label{form24} \nu_{1}(B) = \int_{\{\omega \in \Omega : \gamma(\omega) \cap B \neq \emptyset\}} \calH^{1}(\gamma(\omega) \cap B) \, d\tn(\omega). \end{equation} 
We write $\tn(B) := \tn\{\omega \in \Omega : \gamma(\omega) \cap B \neq \emptyset\}$, and we claim that
\begin{displaymath} \tn(B) \sim \tn(B/2) \sim \beta^{-k} \cdot \diam(B). \end{displaymath}
This follows easily from \eqref{form22}, since every curve $\gamma \in \Gamma$ meeting either $B$ or $B/2$ also intersects $\Omega^{k}$. In fact, the set
\begin{equation}\label{form26} \Omega^{k}(B) = \{x \in \Omega^{k} : x \in \gamma \text{ and } \gamma \cap B \neq \emptyset \text{ for some } \gamma \in \Gamma\} \end{equation}
is a segment of length $\sim \diam(B)$ (and the same holds for $B/2$), so
\begin{displaymath} \tn(B) = \tn(\Omega^{k}(B)) \sim \beta^{-k} \cdot \diam(B) \sim \tn(\Omega^{k}(B/2)) = \tn(B/2) \end{displaymath}
by \eqref{form22}. Since moreover 
\begin{itemize}
\item $\calH^{1}(\gamma(\omega) \cap B) \lesssim \diam(B)$ for all $\omega \in \Omega$, and
\item $\calH^{1}(\gamma(\omega) \cap B) \gtrsim \diam(B)$ for all $\omega \in \Omega$ with $\gamma(\omega) \cap (B/2) \neq \emptyset$,
\end{itemize}
we infer from \eqref{form24} that
\begin{equation}\label{form25} \frac{\nu_{1}(B)}{\calL^{2}(B)} \sim \beta^{-k}. \end{equation}
Writing $z = (s,t)$, we observe that $2^{-k + 1} \leq t \leq 2^{-k}$ whenever $z \in T_{k}$ (simply because this holds for $k = 0$, and $T_{k} = 2^{-k}T_{0}$). Also, $f(z) \sim t^{-\alpha}$, or more precisely
\begin{displaymath} f(z) = (s^{2} + t^{2})^{-\alpha/2} \in [2^{-\alpha/2} \cdot t^{-\alpha},t^{-\alpha}], \end{displaymath}
since $t \geq s$ on $T_{k}$. Note that $2^{-\alpha/2} \in [1/2,1]$ for $\alpha \in [0,1]$, so the implicit constant in $f(z) \sim t^{-\alpha}$ can really be chosen independent of $\alpha$. Now, recalling the choice $\beta = \beta(p) = 2^{-\alpha}$ from under \eqref{betaChoice}, we find from \eqref{form25} that 
\begin{displaymath} \nu_{1}(z) \sim 2^{-\alpha k} \sim t^{-\alpha} \sim f(z) \qquad \text{for } \calL^{2} \text{ a.e. } z \in \bigcup_{k \geq 0} T_{k}. \end{displaymath}
All the implicit constants can, again, be chosen independently of $\alpha \in (0,1)$.

Next, we fix $k \geq 0$ and $z \in \operatorname{int} R_{k}$. Again, we choose $r > 0$ so small that $B := B(z,r) \subset R_{k}$, and we observe that \eqref{form24} holds. The main task is again to find upper and lower bounds for $\tn(B)$. Note that, by construction, every graph $\gamma(\omega) \in \Gamma$ intersecting $B \subset R_{k}$ also intersects $I_{1}^{k} \subset \Omega^{k}$ (with the convention $I_{1}^{0} = I_{1}$ and $\Omega_{0} = \Omega$). Hence, defining $\Omega^{k}(B)$ as before, in \eqref{form26}, we find that 
\begin{displaymath} \tn(B) = \tn(\Omega^{k}(B)) \sim \beta^{-k} \cdot \calH^{1}(\{x \in \Omega^{k} : x \in \gamma \text{ and } \gamma \cap B \neq \emptyset \text{ for some } \gamma \in \Gamma\}),\end{displaymath}
using \eqref{form22} in the last estimate. Combining this with \eqref{form28}, we find that
\begin{displaymath} \tn(B) \sim \beta^{-k} \cdot \diam(B). \end{displaymath} 
This implies \eqref{form25} as before. Finally, we write $z = (s,t)$, and observe that $2^{-k + 1} \leq s \leq 2^{-k}$ for all $z \in R_{k}$, and also that $f(s,t) \sim s^{-\alpha}$ for all $(s,t) \in R_{k}$ (because $s \geq t/2$). Consequently, 
\begin{displaymath} \nu_{1}(z) \sim 2^{-\alpha k} \sim s^{-\alpha} \sim f(z) \qquad \text{for } \calL^{2} \text{ a.e. } z \in \bigcup_{k \geq 0} R_{k}. \end{displaymath}
This completes the proof of \eqref{form15}.

It remains to produce the second representation $(\Omega_{2},\tn_{2},\gamma_{2},\tfrac{d\mu}{d\nu_{2}})$ for $\mu$, which is independent of the first one. Let $M \colon \R^{2} \to \R^{2}$ be a rotation by $\pi/2$ (clockwise, say), and consider the push-forward measures $M(\mu)$ and $M(\nu)$. Note that $f \circ M = f$, so $M(\mu) = \mu$. It follows that from this and \eqref{form15} that
\begin{displaymath} \mu = M(\mu) \in L^{\infty}(M(\nu)) \quad \text{and} \quad M(\nu) \in L^{\infty}([-1,1]^{2},M(\mu)) = L^{\infty}([-1,1],\mu). \end{displaymath}
On the other hand,
\begin{displaymath} M(\nu) = \int_{\Omega_{1}} \calH^{1}|_{M(\gamma_{1}(\omega))} \, d\tn_{1}(\omega) = \int_{\Omega_{2}} \, \calH^{1}|_{M(\gamma_{1}(\omega))} \, d\tn_{2}(\omega) = \nu_{(\Omega_{2},\tn_{2},\gamma_{2})} =: \nu_{2}, \end{displaymath}
where $\gamma_{2}(\omega) := M(\gamma_{1}(\omega))$. So, we find that $(\Omega_{2},\tn_{2},\gamma_{2},\tfrac{d\mu}{d\nu_{2}})$ is the desired second representation of $\mu$. The proof of Theorem \ref{exThm} is complete.

\appendix

\section{The case of two independent axis-parallel representations}

Here we prove Proposition \ref{axisParallel}. The statement is repeated below:
\begin{proposition} Let $\mu$ be a Radon measure on $\R^{2}$ which has two independent axis-parallel representations $(\R,\tn_{1},\gamma_{x},\tfrac{d\mu}{d\nu_{1}})$ and $(\R,\tn_{2},\gamma_{y},\tfrac{d\mu}{d\nu_{2}})$. If both of them are BoA and BoB on $[0,1)^{2}$, then $\mu|_{[0,1)^{2}} \ll \calL^{2}$ with $\mu(x) \sim \mu([0,1)^{2})$ for $\calL^{2}$ almost every $x \in [0,1)^{2}$.
\end{proposition}

\begin{proof} Note that $\mu|_{[0,1)^{2}} \in L^{2}$ by Theorem \ref{main}. Let $Q_{1},Q_{2} \in \calD_{n}([0,1)^{2})$ be dyadic sub-squares of $[0,1)^{2}$ of side-length $2^{-n}$, $n \geq 0$. Write $Q_{1} := I_{1} \times J_{1}$ and $Q_{2} := I_{2} \times J_{2}$. Then also $Q^{\ast} := I_{1} \times J_{2} \in \calD_{n}([0,1)^{2})$, and
\begin{align*} \mu(Q^{\ast}) & \sim \int_{I_{1}} \calH^{1}(Q^{\ast} \cap (\{\omega\} \times \R)) \, d\tn_{1}(\omega) = r \cdot \tn_{1}(I_{1})\\
& = \int_{I_{1}} \calH^{1}(Q_{1} \cap (\{\omega\} \times \R)) \, d\tn_{1}(\omega) \sim \mu(Q_{1}). \end{align*}
Similarly $\mu(Q_{2}) \sim \mu(Q^{\ast})$, so $\mu(Q_{1})/\calL^{2}(Q_{1}) \sim \mu(Q_{2})/\calL^{2}(Q_{2})$, and consequently
\begin{displaymath} \frac{\mu(Q)}{\calL^{2}(Q)} \sim \sum_{Q' \in \calD_{n}([0,1)^{2})} \frac{\mu(Q')}{\calL^{2}(Q')} \cdot \calL^{2}(Q') = \mu([0,1)^{2}) \end{displaymath}
for all $Q \in \calD_{n}([0,1)^{2})$ and $n \geq 0$. The claim now follows from the Lebesgue differentiation theorem. \end{proof}

\bibliographystyle{plain}
\bibliography{references}


\end{document}